\documentclass[letterpaper, 10 pt, conference]{ieeeconf}  

\IEEEoverridecommandlockouts                              

\usepackage{graphicx}
\usepackage{siunitx}
\usepackage{amsmath} 
\usepackage{amssymb}  
\usepackage{color}
\usepackage{cite}
\usepackage{enumerate}
\DeclareMathOperator{\e}{\mathrm{e}}

\newtheorem{proposition}{Proposition}
\newtheorem{remark}{Remark}
\newtheorem{lemma}{Lemma}

\newtheorem{corollary}{Corollary}
\title{\LARGE \bf Output Corridor Impulsive Control of First-order Continuous System\\ with Non-local Attractivity Analysis}

\author{Alexander Medvedev and Anton V. Proskurnikov
\thanks{Alexander Medvedev [{\tt\small alexander.medvedev@it.uu.se}] is with Department of Information Technology, 
        Uppsala University, SE-752 37 Uppsala, Sweden.}%
\thanks{Anton V. Proskurnikov  [{\tt\small anton.p.1982@ieee.org}] is with the Department of Electronics and Telecommunications, Politecnico di Torino, Turin, Italy, 10129.}%
}

\begin{document}

\maketitle
\thispagestyle{empty}
\pagestyle{empty}

\begin{abstract}
This paper addresses the design of an impulsive controller for a  continuous scalar time-invariant linear plant that constitutes the simplest conceivable model of 
chemical kinetics. The model is ubiquitous
in process control as well as pharmacometrics and readily generalizes to systems of Wiener structure. Given the impulsive nature of the feedback, the control problem formulation is particularly suited to discrete dosing applications in engineering and medicine, where both doses and inter-dose intervals are manipulated. Since the feedback controller acts at discrete time instants and employs both amplitude and frequency modulation, whereas the plant is continuous, the closed-loop system exhibits hybrid dynamics featuring complex nonlinear phenomena. The problem of confining the plant output to a predefined corridor of values is considered. The method at the heart of the proposed approach is to design a stable periodic solution, called a 1-cycle, whose one-dimensional orbit coincides with the predefined corridor. Conditions ensuring local and global attractivity of the 1-cycle are established. As a numerical illustration of the proposed approach, the problem of intravenous paracetamol dosing is considered.

\end{abstract}

\section{INTRODUCTION}
Impulsive feedback  constitutes an attractive  paradigm in applications where corrective action is executed seldom and a discrete (sampled-data) controller with synchronous sampling is excessive or infeasible. Open- and closed-loop impulsive control is widely used in, e.g., discrete dosing of chemicals and drugs, mechanical systems with impacts \cite{KZM25}, spacecraft control \cite{SLG21}, temperature regulation, and ecological and environmental systems~\cite{HCN14}. More spectacular application examples are impulsive altitude regulation in hot air balloons by means of ``burner duty cycle" and the actuation of pumped hydrofoils through load-unload body movement.  Because the plant evolves in an open loop between impulses, impulsive control is most effective for systems whose dynamics are sufficiently slow over the inter-impulse intervals.

Since process control traditionally aims at stabilizing the measured output at a predefined setpoint, the literature on impulsive control has naturally emphasized stability and dissipativity properties, whereas the design of sustained periodic regimes has received less attention. 
Periodic solutions are typically characterized by fixed points of a discrete (Poincar{\'e} or return) map~\cite{HCN14,GC20}.
Then designing a feedback that sustains a desired periodic solution corresponds to the stabilization of a fixed point.  

 In process control, 90--95\% of all plant models are linear of first order, possibly being augmented with a static nonlinearity. A similarly, in pharmacometry, pharmacokinetics (PK) of more than 90\% of drugs are adequately described by a single-compartment model with elimination of first-order~\cite{EW07}. Naturally, the pharmacodynamics (PD) are nonlinear, and generally, well captured  by Hill functions. 

 Studying scalar continuous dynamics under pulse-modulated feedback, which gives rise to a hybrid system of effective order two, significantly simplifies the mathematical setup without rendering the problem trivial. The nonlinear dynamics of such hybrid systems were investigated in~\cite{ZCM12a}, where a wide range of complex phenomena were demonstrated, including high-period cycles and deterministic chaos. The model in~\cite{ZCM12a} stems from studies of the so-called Impulsive Goodwin's Oscillator (IGO), which was motivated by endocrine regulation models~\cite{Aut09} and builds on the continuous Goodwin's oscillator~\cite{Good65}, a widespread model in mathematical biology. The IGO model has been generalized to systems with an arbitrary number of compartments; however, the relevant works~\cite{Churilov2020,PRM24} mainly focus on the existence and local orbital stability of special periodic solutions (1-cycles). Low-order dynamics in the single-compartment case not only introduce particularities in the nonlinear dynamics analysis~\cite{ZCM12a} but also open the perspective of solving some impulsive control design problems explicitly.
 



An example of such a problem, important for drug dosing, is the design of a controller that maintains the output of the plant within a predefined corridor. The standard approach to solving this problem for a general dosing system is model predictive control~\cite{Ntouskas21}. An alternative method, proposed in~\cite{MPZ24}, is based on exciting and stabilizing a special periodic solution of the impulsive system, called 1-cycle. The approach in~\cite{MPZ24} is based on the properties of the three-dimensional IGO (in particular, on the extreme points of its impulse response) and cannot be directly applied to the simplest single-compartment dynamics. Furthermore, in the general case, only local stability of the 1-cycle is guaranteed. 

\paragraph*{Contributions} The contributions of this paper are as follows. First, we formulate the output corridor control problem for scalar (single-compartment) kinetics and exploit its special structure to obtain an analytic expression for the periodic solution that tightly satisfies the corridor constraint -- that is, the orbit endpoints coincide with the corridor bounds. Second, we derive conditions for the global attractivity of this solution. Third, we show that a proper design of the impulsive feedback guarantees super-exponential convergence in a neighborhood of the solution, which in practice is indistinguishable from finite-time convergence. Finally, we apply the results to paracetamol dosing and illustrate the theoretical findings by numerical simulations.



\section{PROBLEM FORMULATION}\label{sec.igo}

Consider the scalar single-compartment kinetic model
\begin{equation}                            \label{eq:1}
\dot{x}(t) = ax(t), \quad x(0)>0,  \quad a<0,
\end{equation}
that is Hurwitz stable and positive: $x(t)=\e^{ta}x(0)>0$ for all $t$. To achieve the desired closed-loop behavior, a pulse-modulated feedback is introduced in which  the impulse weights and impulse timing are determined by the measured continuous plant state $x(t)$:
\begin{align}\label{eq:2}
x(t_n^+) &= x(t_n^-) +\lambda_n , \quad                                          
t_{n+1} =t_n+T_n,
\\ 
T_n &=\Phi(x(t_n^-)), \; \lambda_n= F(x(t_n^-)), \quad n=0,1,\ldots \notag
\end{align}
 In theory of pulse-modulated systems~\cite{GC98}, $F(\cdot)$ is called the amplitude modulation function and $\Phi(\cdot)$ is referred to as the frequency modulation function. The minus and plus in a superscript in~\eqref{eq:2} denote the left-sided and
a right-sided limit, respectively. Notice  that the continuous-time state $x(t)$ has discontinuities due to  the jumps in \eqref{eq:2}. The instants $t_n$ are called (impulse) firing times and $\lambda_n$ represents the corresponding impulse weight. The equation for the firing times is a difference equation of first order and contributes discrete dynamics to the closed loop. 

To explicitly restrict the domain in which the solutions of  closed-loop system~\eqref{eq:1},~\eqref{eq:2} ultimately evolve,  boundedness of the modulation functions is 
required 
\begin{equation}                             \label{eq:2a}
0<\Phi_1\le \Phi(\cdot)\le\Phi_2, \quad 0<F_1\le F(\cdot)\le F_2,
\end{equation}
where $\Phi_1$, $\Phi_2$, $F_1$, $F_2$ are constants. Under these limitations, all solutions asymptotically arrive at an interval that can be computed explicitly~\cite{ZCM12a}.

The modulation functions are assumed to be locally Lipschitz, and hence differentiable almost everywhere and absolutely continuous on each finite interval. In the standard IGO design, the monotonicity conditions are imposed~\cite{ZCM12a,PRM24} requiring that $F(\cdot)$ is \emph{non-increasing} and $\Phi(\cdot)$ is \emph{non-decreasing}; accordingly, the controller~\eqref{eq:2} realizes negative feedback from the continuous output to both the pulse amplitude and the pulse frequency. Specifically, an increased value of $x(t_n^-)$ leads to a decreased (or unchanged) weight $\lambda_{n}$ of the instantaneous impulse. Furthermore, the interval $t_{n+1}-t_n$ before the next impulse increases, so that the pulse sequence becomes sparser. This negative feedback mechanism keeps the controlled output in the vicinity of the stationary solution while operating exclusively through positive signals. As shown below, this feedback property is not necessary for the stabilization of the periodic orbit, although it ensures some additional features of the closed-loop system. We thus present the theory separately for general controller~\eqref{eq:2} and for the negative feedback.


The modulation functions $F(\cdot)$, $\Phi(\cdot)$ constitute  the degrees of freedom in the impulsive controller design and are selected to ensure certain desired properties of closed-loop system~\eqref{eq:1},~\eqref{eq:2}. As discussed in Section~\ref{sec.stab}, the time of the first impulse $t_0\geq 0$ is an additional design parameter that can be chosen in an event-triggered manner; in some situations, this allows to achieve the control objective in finite time. We suppose that, if $t_0>0$, the state evolves according to~\eqref{eq:1} on the initial interval $(0,t_0)$. 

In the present study, the following control problem is addressed.
\paragraph*{\bf Output corridor impulsive control problem}  
Given the plant~\eqref{eq:1},
design the modulation functions $\Phi(\cdot)$ and $F(\cdot)$ of the impulsive feedback controller~\eqref{eq:2} that satisfy~\eqref{eq:2a} and maintain the steady-state values of $x(t)$ within the corridor 
\begin{equation}\label{eq:corridor}
    x(t)\in \lbrack x_{\min}^*,x_{\max}^*\rbrack, \quad 0<x_{\min}^*<x_{\max}^*.
\end{equation}

The solution scheme for this problem, presented in the subsequent sections, is as follows. We first design a special periodic solution, termed a 1-cycle, that tightly satisfies the corridor condition~\eqref{eq:corridor}. We then establish conditions guaranteeing local and global attractivity of the 1-cycle. Solutions attracted by the 1-cycle satisfy~\eqref{eq:corridor} at steady state.

\section{Background results}

In this section, we collect several technical results used throughout the paper. Henceforth we assume that $F,\Phi$ satisfy global bounds~\eqref{eq:2a} but yet are not necessarily monotone.

\subsection*{ Ultimate bounds on the solutions}
The solution to~\eqref{eq:1} is monotonically decreasing between consecutive impulses. Indeed, for $t\in(t_{n-1},t_n)$, one has
\begin{gather}
x(t)=\e^{a(t-t_n)}x_n\quad \text{whence}\notag\\
x_n < x(t) < \e^{-a(t_n-t_{n-1})}x_n=\e^{-a\Phi(x_{n-1})}x_n.\label{eq:x-via-xn}
\end{gather}
As a consequence, the following ultimate bounds on the solutions can be derived.
\begin{proposition}[\cite{ZCM12a}]\label{prop.bound} For any admissible initial conditions $x(0)>0,t_0\geq 0$, the solution to \eqref{eq:1} under  pulse-modulated feedback  \eqref{eq:2} is ultimately bounded from above and below
\begin{align}
    \lim_{t\to\infty}\inf x(t)&\ge \frac{F_1}{\e^{-a\Phi_2}-1},\label{eq:upper-old}\\
    \lim_{t\to\infty}\sup x(t)&\le \frac{F_2}{1-\e^{a\Phi_1}}\label{eq:lower-old}.
\end{align}
\end{proposition}
\vskip1mm

As will be shown below in Lemma~\ref{lem.bound}, the estimates from Proposition~\ref{prop.bound} can be substantially tightened in the case of the one-dimensional IGO with monotone nonlinearities~\cite{ZCM12a}. In view of Proposition~\ref{prop.bound}, the bounds for the modulation functions should be chosen compatible with the corridor~\eqref{eq:corridor}:
\[
 \frac{F_1}{\e^{a\Phi_2}-1}\le x_{\min}^*<x_{\max}^*\le \frac{F_2}{1-\e^{-a\Phi_1}},
\]
since otherwise the full corridor cannot be utilized.

Another immediate consequence of~\eqref{eq:x-via-xn} is that, under  controller~\eqref{eq:2}, one has $t_n\to\infty$, and thus the corridor condition~\eqref{eq:corridor} at steady state is equivalent to the relation
\begin{equation}\label{eq:corridor-t_n}
x_{\min}^*\leq\liminf_{n\to\infty}x(t_n^-)\leq \limsup_{n\to\infty}x(t_n^+)\leq x_{\max}^*.
\end{equation}

\subsection*{The return map and 1-cycles}

The satisfaction of the corridor condition thus depends only on the sequence of sampled (pre-impulse) states 
\[
x_n\triangleq x(t_n^-),\quad n=0,1,\ldots
\] 
which, in accordance with~\eqref{eq:1},~\eqref{eq:2}, obey the recursion
\begin{equation}\label{eq:discrete}
    x_{n+1}=Q(x_n),
\end{equation}
where the \textbf{return map} is defined as follows:
\begin{equation}\label{eq:Q}
Q(x)\triangleq\e^{a\Phi(x)}(x+F(x)).
\end{equation}

Recall that $\Phi$ and $F$ are continuous and satisfy~\eqref{eq:2a}, entailing that $Q(0)>0$ and $Q(x)<x$ for sufficiently large $x>0$.
Hence, $Q$ admits a fixed point (possibly, non-unique)
\begin{equation}\label{eq:fp}
    x_*=Q(x_*)>0.
\end{equation}
The fixed point gives rise to a periodic solution of system~\eqref{eq:1},~\eqref{eq:2} with $x_n=x(t_n^-)\equiv x_*$, $x(t_n^+)\equiv x_*+\lambda$ and $t_{n+1}-t_n\equiv T$, where $T\triangleq\Phi(x_*)$ and $\lambda\triangleq F(x_*)$.
Such a periodic solution to system~\eqref{eq:1},~\eqref{eq:2}, with only one impulse fired by the pulse-modulated feedback in the least period, is called a \textbf{1-cycle}. As discussed in~\cite{ZCM12a,Churilov2020}, any 1-cycle corresponds to a fixed point of return map~\eqref{eq:fp}.

Note that the period $T=\Phi(x_*)$ and the impulse weight $\lambda=F(x_*)$ determine the 1-cycle uniquely, since $x(t)=\e^{a(t-t_n)}x_*$ for $t\in (t_n,t_{n+1})$, and equation~\eqref{eq:fp} yields 
\begin{equation}\label{eq:fp_an}
x_*=\frac{\lambda \e^{aT}}{1-\e^{aT}}.
\end{equation}
On the other hand, for an arbitrary pair $\lambda,T>0$, the state in~\eqref{eq:fp_an} serves as a fixed point~\eqref{eq:fp} under any modulation functions satisfying the interpolation constraints $T=\Phi(x_*)$, $\lambda=F(x_*)$. 
For brevity, we call the pair $(T,\lambda)$ the \textbf{parameters} of the 1-cycle.

\subsection*{The corridor-spanning 1-cycle} 

In the output corridor impulsive control problem, the parameters are calculated from the corridor bounds and specify one point (i.e., the fixed point in~\eqref{eq:fp_an}) on the modulation functions $F,\Phi$, ensuring the desired steady-state behavior. The procedure follows immediately from the result below.
\begin{proposition}\label{prop:tight-1cycle}
    Let closed-loop system \eqref{eq:1}, \eqref{eq:2} exhibit a 1-cycle with the parameters $(T,\lambda)$. Then, over each period of this 1-cycle, the state lies between the values
    \[
    \inf x(t)=x_*=\frac{\lambda \e^{aT}}{1-\e^{aT}}, \quad \sup x(t)=x_*+\lambda=\frac{\lambda} {1-\e^{aT}}.
    \]
\end{proposition}

\begin{proof}
The proof is straightforward by applying  inequalities~\eqref{eq:x-via-xn} to the 1-cycle with parameters $(T,\lambda)$ and using the fact that $x_n=x_*$ and $x(t_n^+)=x_n+F(x_n)=x_*+\lambda$, where $x_*$ is specified by~\eqref{eq:fp_an}.
\end{proof}

Given corridor~\eqref{eq:corridor}, a 1-cycle is called~\textbf{corridor-spanning} if its orbit spans the corridor, that is, if its parameters satisfy the equations:
\begin{equation}\label{eq:key-equation}
x_{\min}^*=\frac{\lambda \e^{aT}}{1-\e^{aT}},\quad x_{\max}^*=\frac{\lambda} {1-\e^{aT}}.
\end{equation}

Obviously, equations~\eqref{eq:key-equation} are uniquely solvable for any $a<0$ and every pair $x_{\max}^*>x_{\min}^*>0$:
\begin{align*}
    T=\frac{\ln x_{\min}^*-\ln x_{\max}^*}{a},\quad
    \lambda = x_{\max}^*- x_{\min}^*.
\end{align*}

\section{Global and Local Stability of 1-cycle}\label{sec.stab}

In this section, we investigate the global and local attractivity properties of the corridor-spanning 1-cycle. Note that convergence to the 1-cycle is fully determined by the sequence of sampled states $x_n=x(t_n^-)$.
\begin{proposition}\label{prop:converge}
Consider a solution of closed-loop system~\eqref{eq:1},~\eqref{eq:2}, and assume that $x_n\to x_*=Q(x_*)$, where the fixed point $x_*$ corresponds to the 1-cycle with parameters $(T,\lambda)$. Then, the following statements hold as $n\to\infty$:
\begin{enumerate}[(i)]
\item the post-impulse state $x(t_n^+)$ converges to $x_*+\lambda$;
\item the intervals between consecutive impulses $t_{n+1}-t_n$ converge to $T$;
\item for every intermediate time $0<\tau<T_*$, one has\footnote{Notice that, in this case, $t_n+\tau<t_{n+1}$ for sufficiently large $n$ in view of statement (ii).}
\[
x(t_n+\tau)\xrightarrow{} \e^{a\tau}(x_*+F(x_*)).
\]
\end{enumerate}
In particular, if the 1-cycle is corridor-spanning, then the solution satisfies the corridor condition in~\eqref{eq:corridor} at steady state.
\end{proposition}
\begin{proof}
The proof of (i)--(iii) follows in a straightforward way from equations~\eqref{eq:1},~\eqref{eq:2} and the continuity of the modulation functions $\Phi,F$. 
To prove the last statement, it suffices to notice that~\eqref{eq:corridor-t_n} is entailed by (i) and Proposition~\ref{prop:tight-1cycle}.
\end{proof}

If the properties from Proposition~\ref{prop:converge} hold, we say that the solution is \textbf{attracted} (or \textbf{converges})  to the 1-cycle.

\subsection*{Global convergence via the contraction property}

The following lemma is a straightforward consequence of the Banach contraction principle~\cite[Theorem~9.23]{Rudin1976}.
\begin{lemma}\label{lem:glob}
Let there exist a constant $q\in [0,1)$ such that 
\begin{equation}\label{eq:contraction}
|Q'(x)|\leq q<1\quad\text{for almost all}\quad x\geq 0. 
\end{equation}
Then the 1-cycle in closed-loop system~\eqref{eq:1},~\eqref{eq:2} is unique and attracts all other admissible solutions. Furthermore, the convergence is exponential: $|x_n-x_*|=O(q^n)$ as $n\to\infty$.
\end{lemma}
\begin{proof}
Since $\Phi,F$ are locally Lipschitz, so is the mapping $Q$. Hence, $Q$ is absolutely continuous on any interval $[x_1,x_2]$, whence
\[
|Q(x_2)-Q(x_1)|\leq\int_{x_1}^{x_2}|Q'(x)|\,dx\leq q(x_2-x_1),
\]
that is, $Q$ is a $q$-contraction on $[0,\infty)$. The Banach contraction principle entails  that  recursion~\eqref{eq:discrete} converges exponentially to $x_*$
as $n\to\infty$; in particular, $x_*$ is the unique fixed point on $[0,\infty)$. Now, in virtue of Proposition~\ref{prop:converge} every solution converges to the 1-cycle, 
associated with $x_*$.
\end{proof}

\subsection*{Exponential and super-exponential local convergence}

It is well known that the local exponential stability of the 1-cycle reduces to the local exponential stability of the fixed point~\cite{Churilov2020}, that is, to the condition $|Q'(x_*)|<1$. It should be noted, however, that the general  results for multidimensional case ensure only orbital stability and rely on differentiability of $Q$ in a neighborhood of $x_*$. 
\begin{proposition}\label{prop:local-exp}
Assume that  contraction condition~\eqref{eq:contraction} holds on some interval $I=[x_*-\delta,x_*+\delta]$. Then $I$ is invariant under  discrete-time dynamics~\eqref{eq:discrete}, and any solution with $x_0=x(t_0^-)\in I$ is attracted to the 1-cycle. Moreover, the convergence is exponential.
\end{proposition}
\begin{proof}
The invariance is straightforward, since $x\in I$ entails that $|x-x_*|\leq\delta$, whence
\[
|Q(x)-x_*|=|Q(x)-Q(x_*)|\leq q|x-x_*|=q\delta<\delta.
\]
Hence, $Q(x)\in I$, i.e., $I$ is invariant under $Q$, and therefore $x_n\in I$ for all $n$. By the Banach contraction principle, $x_n\to x_*$ exponentially, i.e., the solution is attracted to the 1-cycle in accordance with Proposition~\ref{prop:converge}.
\end{proof}

Note that if $Q$ is continuously differentiable in a neighborhood of $x_*$ and $|Q'(x_*)|<1$, the condition of Proposition~\ref{prop:local-exp} can always be satisfied by choosing $\delta$ small, since $Q'(x)\to Q'(x_*)$ as $x\to x_*$. Proposition~\ref{prop:local-exp} thus gives an estimate, possibly conservative, of the 1-cycle's basin of attraction.

The fastest convergence rate corresponds to the case where $Q$ is ``flat'' at the point $x_*$, i.e., $Q'(x_*)=0$. If $Q$ is $C^2$-smooth on the interval $I$, the residual term in the Taylor expansion $Q(x)-x_*=Q(x)-Q(x_*)-Q'(x_*)(x-x_*)$ admits an explicit estimate
\[
|Q(x)-x_*|\leq \frac12\sup_{\xi\in I}|Q''(\xi)|(x-x_*)^2.
\]
This motivated the following proposition, establishing ultrafast super-exponential\footnote{In the numerical methods literature, this type of convergence is termed ``quadratic'', while standard exponential convergence is termed ``linear.''} convergence in the vicinity of $x_*$.

\begin{proposition} \label{prop:super-exp}
Assume that $Q'(x)=0$ and that $\delta>0$, $\alpha>0$ are chosen so that the  inequality  below holds: 
\[ 
|Q(x)-Q(x_*)|\leq \alpha(x-x_*)^2<\infty\quad\forall x\in I\triangleq[x_*-\delta,x_*+\delta]. 
\] 
For every $r\leq\min(\delta,\alpha^{-1})$, the interval $I_0\triangleq [x_*-r,x_*+r]$ is then invariant under $Q$, and one has 
\begin{equation}\label{eq.super-exp} 
|x_n-x_*|\leq \alpha^{-1}(\alpha|x_0-x_*|)^{2^n},\;\;\forall x_0\in I_0. 
\end{equation} 
\end{proposition} 
\begin{proof} The proof of invariance is similar to Proposition~\ref{prop:local-exp}. Since $\alpha r\leq 1$ and $r\leq \delta$, for all $x\in I_0$ one has 
\[ 
|Q(x)-x_*|\leq \alpha|x-x_*|^2\leq\alpha r^2\leq r, 
\] 
that is, $Q(x)\in I_0$. Thus, if $x_0\in I_0$, then $x_n\in I_0$ for all $n$. Now~\eqref{eq.super-exp} is proved by induction on $n$. The base case $n=0$ is obvious. Assuming that~\eqref{eq.super-exp} holds for $n$ and denoting for brevity $\delta_0=|x_0-x_*|$, one has 
\[ 
|x_{n+1}-x_*|=|Q(x_n)-x_*|\leq\alpha|x_n-x_*|^2\leq\alpha\cdot\alpha^{-2}(\alpha\delta_0)^{2^{n+1}},
\] which is equivalent to~\eqref{eq.super-exp} with $n$ replaced by $n+1$. 
\end{proof}

\subsection*{Choice of initial impulse time}

In drug dosing applications (see Section~\ref{sec.simul}), drug administration is typically initiated by a bolus dose calculated based on a combination of patient-specific factors (weight, age, organ function) and independent of patient response. 
Control algorithm~\eqref{eq:2} fires the first impulse at time $t_0$ that can be regarded as an additional design parameter. For instance, if $x(0)>x_*$, the controller can wait until the solution of~\eqref{eq:1} reaches $x_*$. In reality, due to  model and measurement  uncertainty, $x(t)$ is in a small neighborhood of $x_*$.

In the matter of fact, the distinction between local and global attractivity is negligible, as one can guarantee that $x_0=x(t_0^-)$ is close to $x_*$ (see the Case~1 in Section~\ref{sec.simul}). Disregarding measurement and modeling errors, one can formally guarantee finite-time (deadbeat) convergence to the 1-cycle.
However, medical regimens often impose a maximum time interval between consecutive doses (naturally captured by the bound $\Phi_2$ in~\eqref{eq:2a}), which the initial time $t_0$ cannot exceed. Therefore, one cannot initiate the pulse-modulated feedback at the point $x_*$, and the convergence rate to the 1-cycle should be considered.
\color{black}

\section{The Negative Feedback Case}\label{sec.negative}

In this section, we consider the one-dimensional IGO system introduced in~\cite{ZCM12a}, characterized by a monotonically non-increasing $F$ and a monotonically non-decreasing $\Phi$. As discussed previously, this choice of modulation functions implements negative feedback from the state to both the amplitude and frequency of the pulses.

This special case enjoys additional properties; in particular, it is possible to substantially tighten  the bounds in Proposition~\ref{prop.bound}, which in some situations leads to an alternative global attractivity criterion.

\subsection*{A tightened estimate of the solution}

Introduce the function
\[
\Psi(x)\triangleq\frac{F(x)\e^{a\Phi(x)}}{1-\e^{a\Phi(x)}}=\frac{Q(x)-\e^{a\Phi(x)}x}{1-\e^{a\Phi(x)}}.
\]

Since $F$ is decreasing, while the assumptions $a<0$ and $\Phi$ increasing 
imply that $\e^{a\Phi(x)}$ is (monotonously) decreasing and $1-\e^{a\Phi(x)}>0$ is increasing, it follows that $\Psi$ is 
decreasing on $[0,\infty)$. Hence, the composition $\Psi\circ\Psi$ is continuous and increasing.

The set of solutions to the equation
\begin{equation}\label{eq:psi2}
x=\Psi(\Psi(x)),\qquad x\ge 0,
\end{equation}
is closed. Moreover, $\Psi$ maps $[0,\infty)$ into a bounded interval:
\[
\Psi:[0,\infty)\to [\Psi(\infty),\Psi(0)]\subseteq 
\left[
\frac{F_1\e^{a\Phi_2}}{1-\e^{a\Phi_2}},\,\frac{F_2\e^{a\Phi_1}}{1-\e^{a\Phi_1}}
\right],
\]
where $F_1,F_2$ and $\Phi_1,\Phi_2$ are the bounds introduced in~\eqref{eq:2a}, and $\Psi(\infty)=\lim_{x\to\infty}\Psi(x)$.
Hence, all solutions of~\eqref{eq:psi2} are contained in a bounded and strictly positive set. 
Therefore, the solution set of~\eqref{eq:psi2} is compact, and in particular
the minimal and maximal solutions exist. Denote them by $m_*>0$ and $M_*\geq m_*$,
respectively.

The following lemma  provides a simple iterative procedure for 
determining $m_*$ and $M_*$ as well as clarifies their relation to the solutions 
of the IGO.
\begin{lemma}\label{lem.bound}
Suppose that $\Phi(\cdot)$ is non-decreasing and $F(\cdot)$ is non-increasing.
The minimal and maximal solutions $m_*$ and $M_*$ to~\eqref{eq:psi2}
can be obtained as limits of the monotone sequences $\{m_n\}$ and $\{M_n\}$ defined by
\begin{equation*}
\begin{gathered}
m_0=0, \qquad m_n=(\Psi\circ\Psi)(m_{n-1}), \quad n\ge 1,\\
M_0=\Psi(0),\;\; M_n=(\Psi\circ\Psi)(M_{n-1})=\Psi(m_n),\;\; n\ge 1.
\end{gathered}
\end{equation*}
The sequence $\{m_n\}$ is increasing and converges to $m_*$, 
whereas $\{M_n\}$ is decreasing and converges to $M_*$. Hence, $M_*=\Psi(m_*)$ and $m_*=\Psi(M_*)$.
Also, $M_n>m_n$ for all $n$.

Finally, the interval $[m_*,M_*]$ is a global attractor for discrete-time system~\eqref{eq:discrete}: for every initial condition $x_0\geq 0$,
\[
m_*\leq\liminf_{n\to\infty}x_n\leq\limsup_{n\to\infty}x_n\leq M_*.
\]
Consequently, every solution of the hybrid dynamics satisfies the ultimate bounds
\begin{equation}\label{eq:bounds_new}
    m_*\le \liminf_{t\to\infty} x(t)
\le \limsup_{t\to\infty} x(t)
\le \e^{-a\Phi(M_*)}\,M_* .
\end{equation}
\end{lemma}

\begin{proof}
Using induction on $n$, we show that the sequence $\{m_n\}$ is increasing.  For $n=1$, we have
$m_1=(\Psi\circ\Psi)(0)>0=m_0$ since $\Psi$ is strictly positive on $[0,\infty)$. Assume that $m_n>m_{n-1}$ for some $n\ge 1$. Since $\Psi\circ\Psi$ is increasing, it follows that
$m_{n+1}=(\Psi\circ\Psi)(m_n)
>(\Psi\circ\Psi)(m_{n-1})= m_n$. Thus $\{m_n\}$ is increasing and, in particular, it has a limit. The proof that $M_n$ is decreasing is analogous, observing that
\[
M_1=\Psi(\Psi(0))<\Psi(0)=M_0,
\]
because $\Psi(0)>0$ and $\Psi$ is decreasing on $[0,\infty)$, and applying the monotonicity of $\Psi\circ\Psi$ to prove the induction step.
A similar induction, starting with the induction step
$M_0=\Psi(m_0)>m_0$, proves that $M_n=\Psi(m_n)>m_n$.
Since $\Psi\circ\Psi$ is continuous, the limits $\bar m\triangleq\lim_{n\to\infty}m_n$ and $\bar M\triangleq\lim_{n\to\infty}M_n$ are its fixed points, i.e., solutions of~\eqref{eq:psi2}. On the other hand, if $x\geq 0$ is another solution to~\eqref{eq:psi2}, then, obviously, $x\leq\Psi(0)=M_0$. 
Using the fact that $m_0\leq x\leq M_0$ and induction on $n$, one shows that
\[
m_n\leq x=(\Psi\circ\Psi)(x)\leq M_n\quad\forall n=0,1,\ldots
\]
that is, $\bar m\leq x\leq\bar M$. In other words, $\bar m=m_*$ and $\bar M=M_*$ are the minimal and the maximal fixed points of $\Psi\circ\Psi$.

To prove the invariance of intervals $[m_n,M_n]$, notice that
\begin{equation}\label{eq:Q-bound}
Q(x)=\e^{a\Phi(x)}x+(1-\e^{a\Phi(x)})\Psi(x),
\end{equation}
lies between the points $x$ and $\Psi(x)$, since $0<\e^{a\Phi(x)}<1$. Hence if $m_n\leq x\leq M_n$, then $Q(x)\in[m_n,M_n]$ in view of
\[
m_n<m_{n+1}=\Psi(M_n)\leq\Psi(x)\leq\Psi(m_n)=M_n.
\]
The interval $[m_*,M_*]$ is forward invariant, being the intersection of forward invariant sets $[m_n,M_n]$.

Consider now a solution $\{x_n\}$ of~\eqref{eq:discrete} and denote $\ell\triangleq\liminf_{n\to\infty}x_n$ and $L\triangleq\limsup_{n\to\infty}x_n$.
Since $x_n=x(t_n^-)$ for an appropriately defined solution of~\eqref{eq:1}, Proposition~\ref{prop.bound} guarantees that $0<\ell\leq L<\infty$.
Notice that, since $\Psi$ is continuous and decreasing, one has
\[
\limsup_{n\to\infty}\Psi(x_n)=\Psi(\ell)\geq \Psi(L)=\liminf_{n\to\infty}\Psi(x_n).
\]
Now, we are going to prove that $\Psi(\ell)\geq L$. Indeed, choose an arbitrary constant $\eta>\Psi(\ell)$ and let $N$ be so large that $\Psi(n)<\eta$ for $n\geq N$.
Then, using~\eqref{eq:Q-bound} and recalling that $\Phi(x_n)\geq\Phi_1$, whence $a\Phi(x_n)\leq a\Phi_1<0$,
\[
\begin{aligned}
x_{n+1}\leq (1-\e^{a\Phi(x_n)})\eta&+\e^{a\Phi(x_n)}x_n=\eta+\e^{a\Phi(x_n)}(x_n-\eta)\\
&\leq \eta+\e^{a\Phi_1}\min\{0,x_n-\eta\}.
\end{aligned}
\]
Taking $\limsup$ leads to
\[
L\leq \eta+\e^{a\Phi_1}\min\{0,L-\eta\},
\]
which inequality can only be valid when $\eta\geq L$. In other words, we have proven the implication $\eta>\Psi(\ell)\Longrightarrow\eta\geq L$, which means that
$\Psi(\ell)\geq L$. A symmetric argument shows that $\Psi(L)\leq\ell$. Therefore, $\Psi\circ\Psi(\ell)\leq\Psi(L)\leq\ell$ and, since $\Psi\circ\Psi(0)>0$,
the interval $[0,\ell]$ thus contains a fixed point of $\Psi\circ\Psi$. The definition of $m_*$ implies that $m_*\leq\ell$. Symmetrically, $\Psi\circ\Psi(L)\geq L$, and hence
$[L,\infty)$ also contains a fixed point of $\Psi\circ\Psi$, entailing that $L\leq M_*$. This proves Lemma~\ref{lem.bound}.
\end{proof}
\color{black}

Note that the unique fixed point $x_*=Q(x_*)\geq 0$, in view of~\eqref{eq:Q-bound}, is also a fixed point of $\Psi$ and, in particular,
it is a solution to~\eqref{eq:psi2}. This leads to the following corollary.
\begin{corollary}
The unique fixed point of $Q$ from~\eqref{eq:fp} satisfies the inequalities $m_*\leq x_*\leq M_*$. In the special case where\footnote{This holds, e.g., in the case of periodic impulses driving the plant in open loop ($\Phi$, $F$ and $\Psi$ are constant), so that~\eqref{eq:psi2} has a unique solution.} $m_*=M_*$, all solutions of the discrete-time system~\eqref{eq:discrete}
converge to the fixed point $x_*=m_*=M_*$, and thus all admissible solutions of~\eqref{eq:1},~\eqref{eq:2} converge to the 1-cycle.   
\end{corollary}

More generally, Lemma~\ref{lem.bound} entails that any periodic solution of~\eqref{eq:discrete} stays in the invariant set $[m_*,M_*]$. 

\subsection*{Upper bound on the derivative and uniqueness of the 1-cycle}


Another useful fact regarding the map $Q$ (observed in~\cite{ZCM12a}) is that, provided the derivatives $\Phi'(x)$ and $F'(x)$ exist, 
\begin{equation}\label{eq:less1}
    Q'(x)=a\Phi'(x)Q(x)+\e^{a\Phi(x)}(1+F'(x))\leq \e^{a\Phi_1}<1,
\end{equation}
in virtue of monotonicity of the two mappings $\Phi'(\cdot)\geq 0$, $F'(\cdot)\leq 0$. This implies the following useful proposition.
\begin{proposition}\label{prop.id-Q}
The function $x-Q(x)$ is strictly increasing on $[0,\infty)$. In particular, the fixed point~\eqref{eq:fp} is unique.
\end{proposition}
\begin{proof}
By assumption, $Q$ is locally Lipschitz. Since $\Phi,F,Q$ are absolutely continuous and differentiable at almost all points $x\geq 0$, 
~\eqref{eq:less1} holds at all $x$ except for a set of Lebesgue zero measure. Furthermore, for every pair $x_2>x_1\geq 0$, the Newton-Leibnitz formula is valid
\[
Q(x_2)-Q(x_1)=\int_{x_1}^{x_2}Q'(x)dx<\int_{x_1}^{x_2}1~\mathrm{d} x=x_2-x_1,
\]
i.e., $x_2-Q(x_2)>x_1-Q(x_1)$, so $x-Q(x)$ is increasing. The last statement is straightforward.
\end{proof}

Finally, one may notice that the global stability condition from Lemma~\ref{lem:glob} in the case of IGO simplifies to the one-sided inequality
$Q'(x)>-q>-1$, since the upper bound is guaranteed by~\eqref{eq:less1}.

\color{black}

\vspace{0.4cm}


\section{Simulation example}\label{sec.simul}
An example illustrating the utility of the proposed control approach in drug dosing is provided in this section. It features a standard pain reliever drug -- intravenous (IV) paracetamol -- an effective analgesic for acute pain management, particularly in postoperative settings, offering a fast-acting alternative when oral route is not feasible.

\subsection{PKPD model}The PK of paracetamol can be approximated~\cite{MJH08} by
\begin{equation}\label{eq:pk}
    \dot x=-k x, \quad x(0)>0,
\end{equation}
where $x$ is drug concentration in \SI{}{mg/L}, $k=\SI{0.28}{h^{-1}}$ is the first-order elimination rate constant, and $x_0$ is the administrated IV bolus dose. The effect of medication is measured on Visual Analogue Scale (VAS), where zero means ``no pain" and ten corresponds to ``worst imaginable pain" and is described by the PD model
\begin{equation}\label{eq:pd}
    y(t)=\varphi(x)\triangleq E_0-\frac{E_{\max} x(t)}{E_{C50}+x(t)},
\end{equation}
where $E_0=10$, $E_{\max}=5.17$, and $E_{C50}=\SI{9.98}{mg/L}$.  In control engineering terms, model \eqref{eq:pk}, \eqref{eq:pd} constitutes a scalar Wiener model.

The recommended dosing of paracetamol for adults is \SI{1000}{mg} three-four times a day, in oral formulation. The cumulative day dose should be kept under  \SI{4}{g} to avoid liver damage. The blood concentration range \qtyrange{10}{20}{mg/L} is considered to be safe. 
Then, in \eqref{eq:corridor}, $x_{\min}^*=10$ and $x_{\max}^*=20$, or, in terms of the measured output, $y_{\min}^*=\varphi(x_{\max}^*)=6.5510$ and $y_{\max}^*=\varphi(x_{\min}^*)=7.4124$. After a typical therapeutic dose, serum levels usually peak below \SI{30}{mg/L}, $\varphi(30)=6.1206$.
Notably, at these levels,  pain should significantly interfere with daily activities and the PD  model in \eqref{eq:pd}  is probably too conservative. 

 A simulation of open-loop paracetamol administration in IV boluses is given in Fig.~\ref{fig:open_loop}. This scenario is in line with Programmed Intermittent Bolus (PIB) technique in which boluses of an anesthetic drug are automatically injected multiple times, with or without patient-controlled boluses, \cite{GAH13}. 
 As seen in Fig.~\ref{fig:open_loop}, without feedback action, the measured output lies quite often outside of the safe corridor, resulting in both underdosing and overdosing episodes.

\begin{figure}[t]
\centering 
\includegraphics[width=0.9\linewidth]{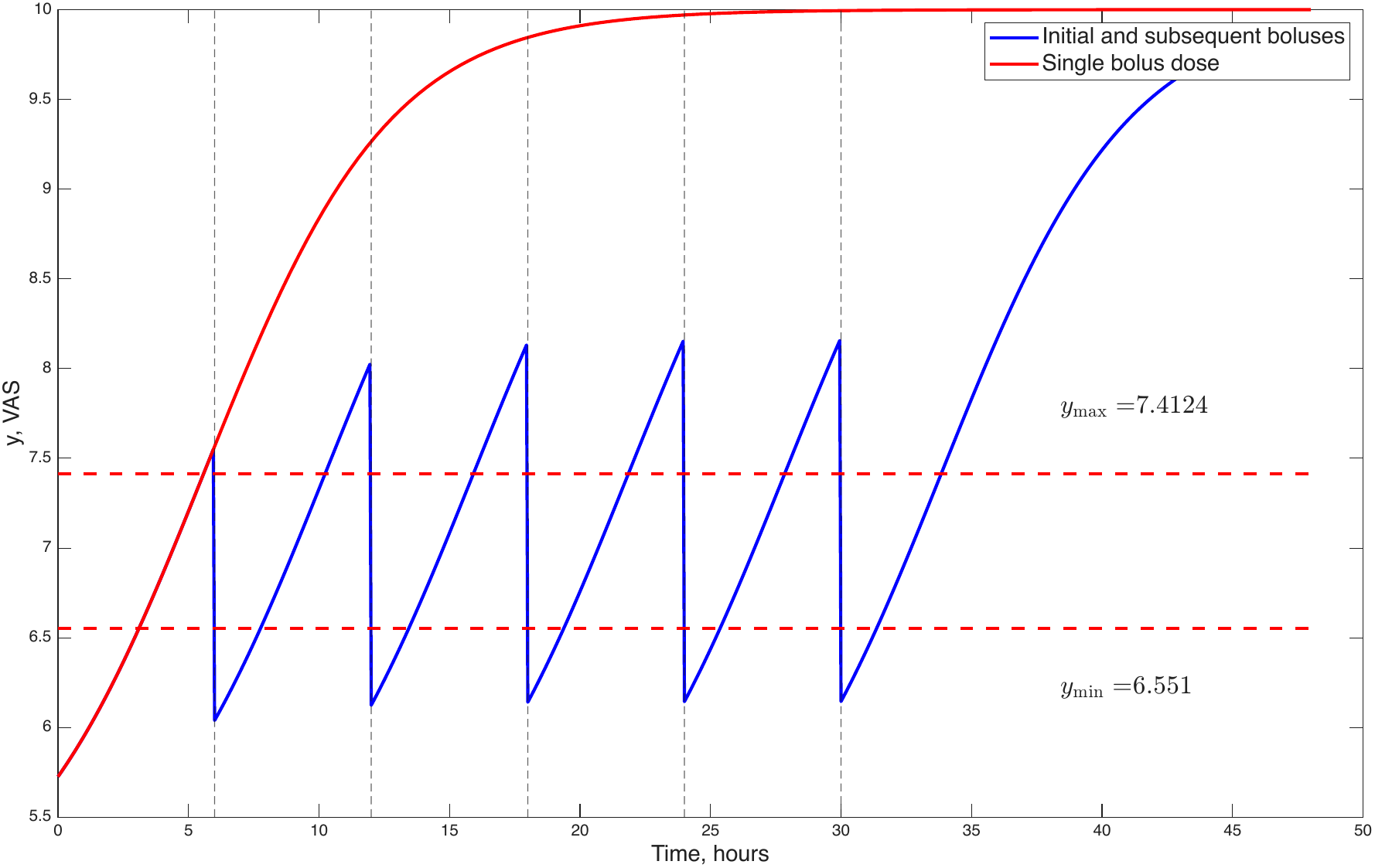}
\caption{ Open-loop dosing of paracetamol. Single  IV bolus dose (in red) vs initial ($\lambda_0=\SI{2000}{mg}$) and five subsequent IV bolus doses ($\lambda_k=\SI{1000}{mg}, k=1,\dots,5$) (in blue). The recommended output value corridor $[y_{\min}^*,y_{\max}^*]$ is marked by red dashed lines. The black vertical lines correspond to the dose administration times, with inter-dose intervals of six hours.
}\label{fig:open_loop}
\end{figure}

\subsection{Controller design}

\paragraph*{Fixed point design} We first find the corridor-spanning 1-cycle by solving equations~\eqref{eq:key-equation}.
Substituting $a=-k$, 
the period and weight of this sought 1-cycle are  
\begin{align*}
    T&=\frac{\ln x_{\min}^*-\ln x_{\max}^*}{a}=\SI{2.4755}{h},\\
    \lambda &= x_{\max}^*- x_{\min}^*=\SI{10}{mg/L},
\end{align*}
and the fixed point is evaluated to $x_*=\SI{10}{mg/L} $. Thus, to comply with the design objective, the following interpolation constraints must be satisfied:
\begin{equation}\label{eq:lambda_T}
    F(x_*)=\lambda, \quad \Phi(x_*)=T.
\end{equation}
The output value corresponding to the fixed point is
\[
y_*\triangleq \varphi(x_*)= 7.4124.
\]
\paragraph*{Stability and convergence}
A straightforward computation gives the derivative of the return map
\begin{equation}\label{eq:Jacobian_fb}
    Q^\prime(x)= \e^{a\Phi(x)}+\e^{a\Phi(x)} \begin{bmatrix}
    1 & a(x+F(x)) \end{bmatrix} \begin{bmatrix}
        F^\prime(x) \\ \Phi^\prime(x)
    \end{bmatrix}.
\end{equation}
Guaranteeing the stability condition $|Q'(x_*)|<1$ is equivalent to solving
a linear static feedback problem, in which the slopes of the modulation functions play the role of controller gains. 
The absolute value $|Q^\prime(x_*)|$ defines the local convergence rate to the 1-cycle. Thus, to achieve fastest possible convergence, i.e., $Q^\prime(x_*)=0$, the modulation functions have to be selected so that 
\begin{equation}\label{eq:deadbeat}
\begin{bmatrix}
    1 & a(x_*+F(x_*)) \end{bmatrix} \begin{bmatrix}
        F^\prime(x_*) \\ \Phi^\prime(x_*)
    \end{bmatrix}=-1.
\end{equation}
Then, from \eqref{eq:deadbeat}, the value of $F^\prime(x_*)$ can be calculated for a given $\Phi^\prime(x_*)$. In particular, picking 
$$\Phi^\prime(x_*)=\Phi^\prime_*\triangleq\frac{-1}{a(x_*+\lambda)}=0.1786,$$ 
results in $F^\prime(x_*)=0$. For affine parametrization of the modulation function used below~\eqref{eq:affine_F}, it means that $F(x)=\mathrm{const}$. 
In the case when $\Phi'(x_*)>\Phi_*'$, one gets $F'(x_*)>0$, thus violating the negative feedback mechanism (Section~\ref{sec.negative}).

\color{black}

\paragraph*{Design of modulation functions}
The drug concentration $x(t)$ is not available for measurement and the VAS rating has to be used for control instead. Then,
 for $F^\prime(\cdot)$ and $\Phi^\prime(\cdot)$  outside of saturation, the chain rule gives
\begin{align}\label{eq:F_prime_Phi_prime}
     F^\prime(x_*)&=\bar F^\prime(y_*)\varphi^\prime(x_*)= k_4 \varphi^\prime(x_*),\\ 
     \Phi^\prime(x_*)&= \bar \Phi^\prime\left(y_*\right)\varphi^\prime(x_*)= k_2 \varphi^\prime(x_*), \nonumber
\end{align}
where $\varphi',\bar\Phi,\bar F$ are found as follows:
\begin{gather}
\varphi^\prime(x)=-\frac{E_{\max} E_{C50}}{(E_{C50}+x)^2},\notag\\
\begin{aligned}
  \bar\Phi (\xi)= \begin{cases} \bar \Phi_2 &  \bar \Phi_2 < k_2\xi +k_1, \\
     k_2\xi +k_1 & \bar\Phi_1 \le  k_2\xi +k_1 \le \bar\Phi_2, \\
  \bar  \Phi_1  &  k_2\xi +k_1 < \bar \Phi_1, 
     \end{cases}
\end{aligned}\label{eq:affine_Phi}\\
\begin{aligned}
    \bar F (\xi)= \begin{cases} \bar F_1 &  k_4\xi +k_3< \bar F_1, \\
     k_4\xi +k_3 & \bar F_1 \le k_4\xi +k_3 \le \bar F_2, \\
    \bar F_2 & \bar F_2 <k_4\xi +k_3,
     \end{cases}
\end{aligned}\label{eq:affine_F}
\end{gather}
where the bounds $\bar F_1=\SI{200}{mg}$, $\bar F_2=\SI{2000}{mg}$, $\bar\Phi_1=\SI{1}{h}$, and $\bar\Phi_2=\SI{8}{h}$ are selected from recommended regimen.
With affine parameterization \eqref{eq:affine_F}, \eqref{eq:affine_Phi} and being evaluated in the fixed point, \eqref{eq:deadbeat} turns into
\begin{equation}\label{eq:k2k4}
    \begin{bmatrix}
    1 & a(x_*+\lambda) \end{bmatrix} \begin{bmatrix}
        k_4 \\ k_2
    \end{bmatrix}\varphi^\prime(x_*)=-1.
\end{equation}
To satisfy \eqref{eq:lambda_T}, the modulation functions are shifted so that 
\begin{align}\label{eq:k1k3}
    F(x_*)&= \bar F ( \varphi(x_*))=k_4 \varphi(x_*)+k_3= \lambda,\\
    \Phi(x_*)&=\bar \Phi (\varphi(x_*))= k_2 \varphi(x_*)+k_1=T. \notag
\end{align}
Now, from a value of $\Phi^\prime(x_*)$, the coefficients $k_i,i=1,\dots,4$ that completely define the pulse-modulated controller can be obtained. In view of \eqref{eq:deadbeat}, the designed controller guarantees optimal (local) convergence rate to the desired 1-cycle. To study the influence of $\Phi^\prime(x_*)$ on the controller performance, consider the following three design cases.
\begin{description}
    \item[Case~1:]\hspace{0.2cm}  $\Phi^\prime(x_*)>\Phi^\prime_*$. Selecting $\Phi^\prime(x_*)=4$, the equation above yields $F^\prime(x_*)= 21.4$. Therefore, \eqref{eq:k2k4}, together with \eqref{eq:k1k3}, give  $k_1=  25.4153$, $k_2 =-3.0948$, $k_3 =132.7279$, and $k_4=-16.5571$. 
    \item[Case~2:]\hspace{0.2cm} $\Phi^\prime(x_*)=\Phi^\prime_*$, $F^\prime(x_*)=0$. The coefficients of the modulation functions are $k_1=3.4996$, $k_2= -0.1382$, $k_3=10$, and $k_4=0$.
    
    \item[Case~3:]\hspace{0.2cm} $\Phi^\prime(x_*)<\Phi^\prime_*$. $\Phi^\prime(x_*)=4$, The coefficients of the modulation functions are $k_1=3.0490$, $k_2= -0.0774$, $k_3= 7.4766$, and $k_4=0.3404$.
\end{description}


 \subsection{Simulation} 
 The performance of the designed pulse-modulated controllers for the PKPD model in \eqref{eq:pk}, \eqref{eq:pd} is assessed in simulation. Similarly to the open-loop case depicted in Fig.~\ref{fig:open_loop},  an initial bolus \SI{2000}{mg} of paracetamol is administered. Afterward, two possibilities exist: either the drug effect  reaches $y_*=\varphi(x_*)$ or the maximum admissible time $\Phi_2=\bar\Phi_2=\SI{8}{h}$ has elapsed since the initial bolus. 
 In the former case, the closed-loop system arrives to the fixed point $x_*$ in one step (in finite time) and, provided $x_*$ is stable, never leaves it. In the latter case, the convergence to the desired periodic solution is asymptotic.


 \paragraph*{Case~1} As seen in Fig.~\ref{fig:closed_loop_case_1}, the feedback controller steers the model output into the desired corridor by means of a single drug dose, exactly, and in a finite time, i.e. in a deadbeat manner.  The reason for the observed deadbeat convergence is that, after the initial bolus and to the first feedback firing, the state variable $x(t)$ follows open-loop dynamics \eqref{eq:pk} until the fixed point $x_*$ is reached. Then, due to \eqref{eq:lambda_T}, the system enters the 1-cycle. Deadbeat performance is preserved for any $x_{\max}^*\le x(0)$ provided the time elapsed for the transition from $x(0)$ to $x(t)=x_*$ is less or equal to $\Phi_2$.
 Notice that despite the coefficients $k_2$ and $k_4$ both being negative, the fixed point is still locally stable (with super-exponential local convergence) as $Q^\prime(x_*)=0$.

\paragraph*{Case~2} This design corresponds to pure frequency modulation feedback as $\bar F(x)=\mathrm{const.}$ Since $\bar\Phi(x)$ is too flat, the controller does not wait for the initial transient to be expired and adds constant drug doses of $\lambda=\SI{10}{mg/L}$ starting from $t_1=\SI{2.7085}{h}$. Convergence to the desired corridor is therefore delayed despite the fact that $Q^\prime(x_*)=0$.

\paragraph*{Case~3} Now the modulation functions agree with the  assumptions made in Section~\ref{sec.igo} and $(\bar\Phi\circ \varphi)(x)$ is increasing while $(\bar F\circ \varphi)(x)$ is decreasing. Both frequency and amplitude modulation are exploited in the controller. Because of that, the convergence of the plant output to the corridor is slightly improved but the slope of the amplitude modulation function is too small to make a noticeable difference. 

\paragraph*{Global convergence} We show the behavior of $Q'$ in Fig.~\ref{fig:Q_prime_x}.
 \begin{figure}[h]
 \centering 
 \includegraphics[width=0.5\linewidth]{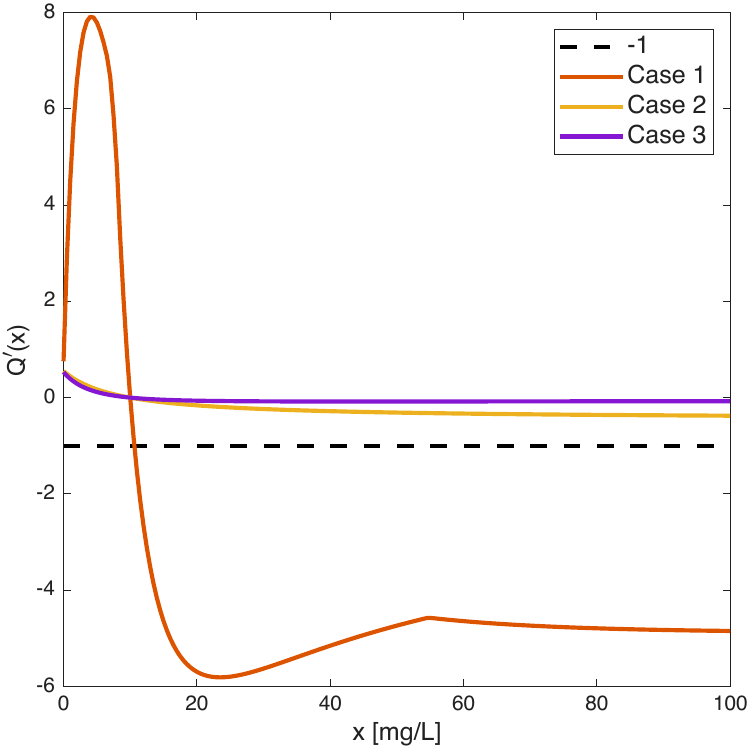}
 \caption{ $Q^\prime(x)$ calculated according to \eqref{eq:Jacobian_fb} for Case~1--Case~3. 
 }\label{fig:Q_prime_x}
 \end{figure}
 In Cases~2 and~3 the global attractivity of 1-cycle is guaranteed by Lemma~\ref{lem:glob}.
 Formally, Lemma~\ref{lem:glob} is inapplicable in Case~1, however, the event-triggered logic of the controller start makes it unnecessary to establish the global convergence as $x_0=x(t_0^-)=x_*$ by design.

\paragraph*{Bounds} All the controllers considered above drive the plant output (the medication effect) to the same predefined corridor corresponding to a stationary solution with drug concentration  $x(t)\in \lbrack 10, 20 \rbrack \SI{}{mg/L}$. It is instructive to compare the actual corridor to theoretically derived bounds. 
Proposition~\ref{prop.bound} ensures the following bounds at steady state:
\[
0.5673 \le x(t) \le 194.9872.
\]
The iterative algorithms of Lemma~\ref{lem.bound} (applicable formally only in Cases~2 and~3) converges to the fixed point $x_*$ in three steps, and yielding, according to \eqref{eq:bounds_new} improved bounds
\[
10 \le x(t) \le 20.
\]
The bounds coincide with the desired corridor values when condition~\eqref{eq:k2k4} is satisfied for the slopes of the modulation functions. In Case~1, the solution, obviously, also satisfies these bounds as it coincides with $1$-cycle for $t\geq t_0$.


 \section*{Conclusions}
The problem of controlling the output of a stable scalar continuous system into a predefined corridor of output values by means of amplitude- and frequency-modulated feedback is considered. The design of the controller is performed by means of constructing a fixed point of a discrete map and guaranteeing its stability. An approach to analyze non-local attractivity of the designed periodic attractor is proposed. Controller performance is illustrated by a simulated drug dosing application using a first-order PKPD model from the literature. Exact convergence of the controller to the predefined output corridor in finite time is demonstrated.

\begin{figure}[t]
\centering 
\includegraphics[width=0.9\linewidth]{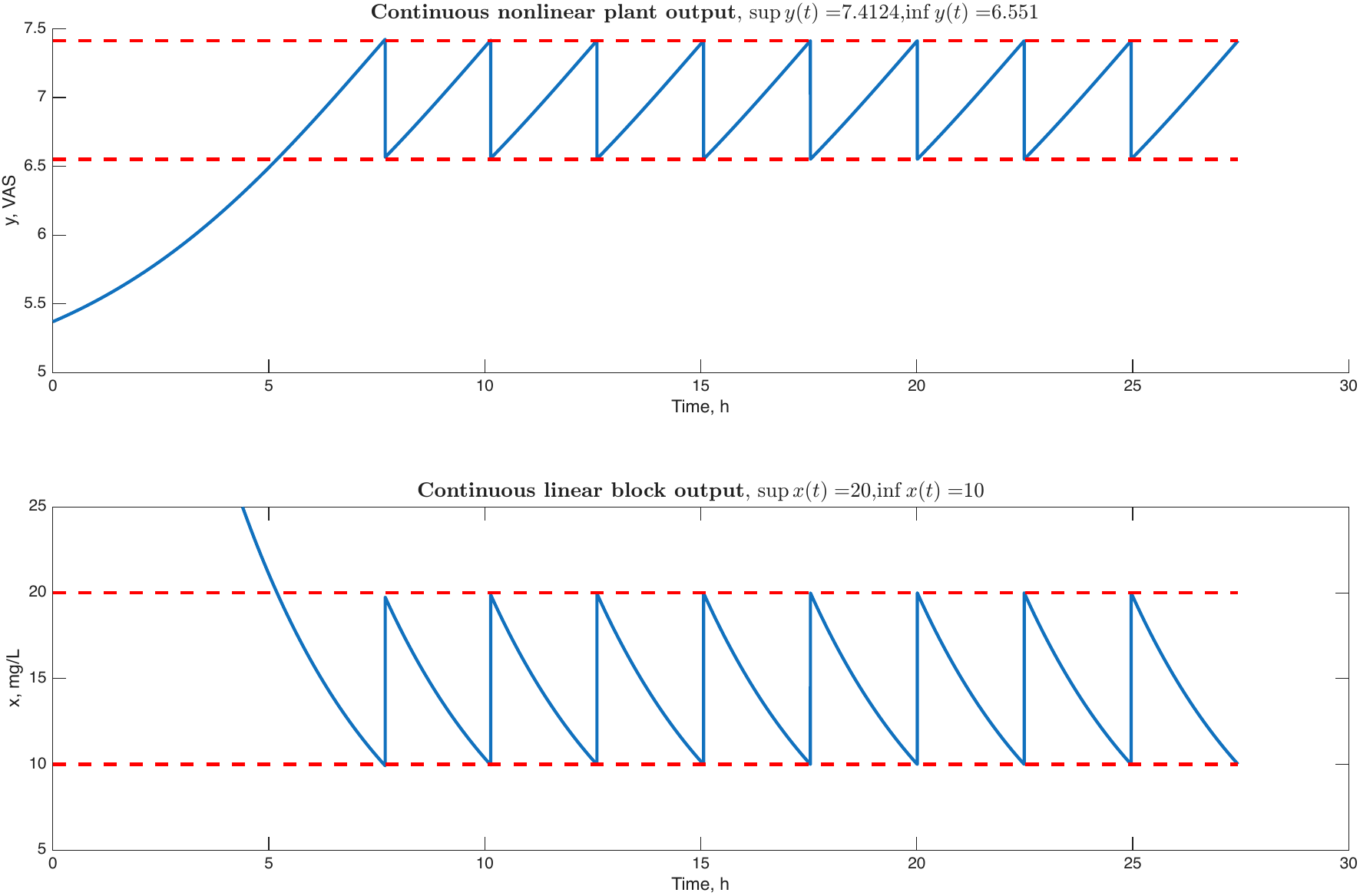}
\caption{ Closed-loop dosing of paracetamol, Case~1. After initial  IV bolus dose ($\lambda_0=\SI{2000}{mg}$), the pulse-modulated feedback drives the drug concentration and medication effect  (in blue) into desired corridor.  Upper plot: the recommended output value corridor $[y_{\min}^*,y_{\max}^*]$ is marked by red dashed lines. Lower plot: the recommended concentration corridor $[x_{\min}^*,x_{\max}^*]$ is marked by red dashed lines.
}\label{fig:closed_loop_case_1}
\end{figure}
\begin{figure}[t]
\centering 
\includegraphics[width=0.9\linewidth]{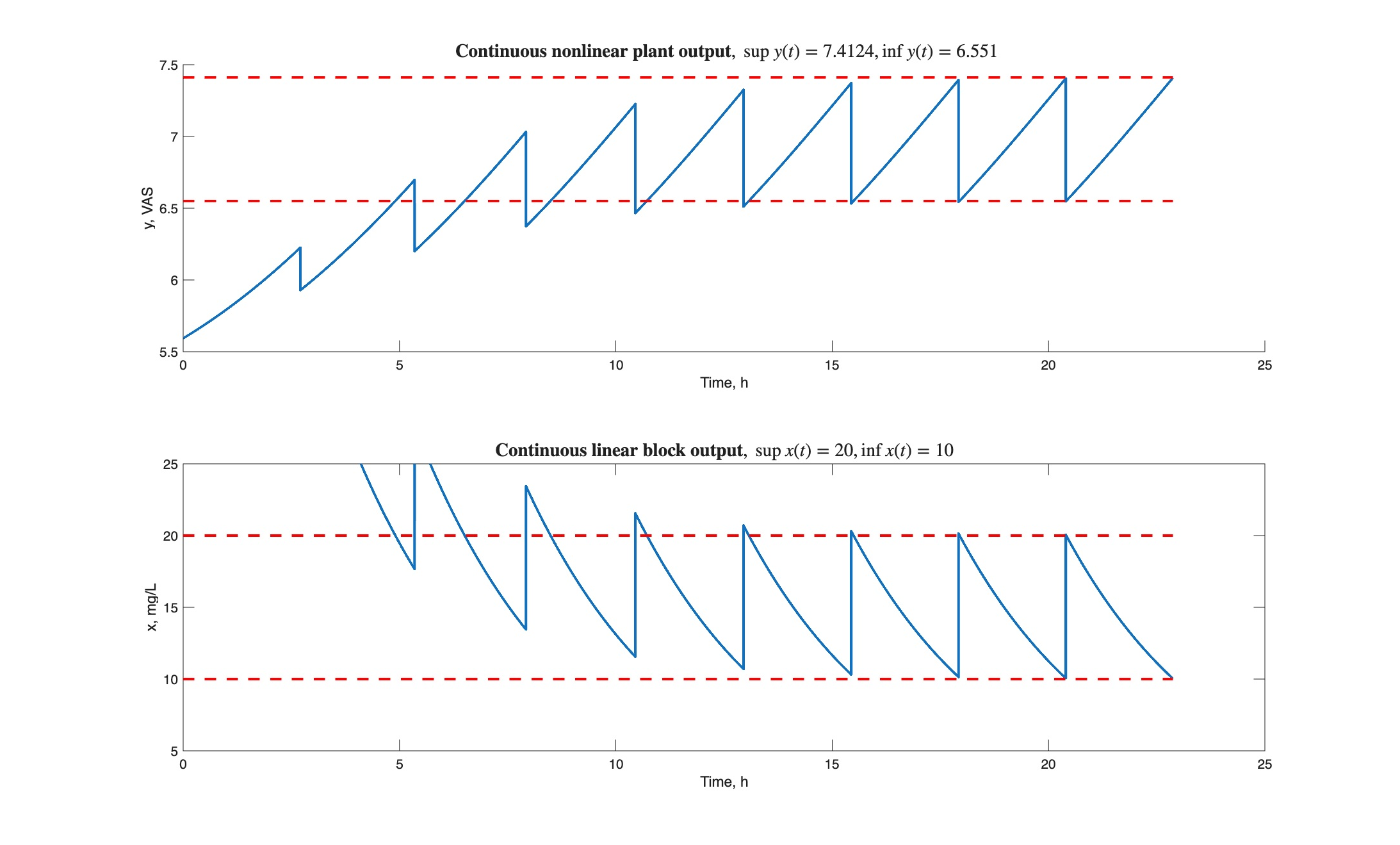}
\caption{ Closed-loop dosing of paracetamol, Case~2. After initial  IV bolus dose ($\lambda_0=\SI{2000}{mg}$), the pulse-modulated feedback drives the drug concentration and medication effect  (in blue) into desired corridor.  Upper plot: the recommended output value corridor $[y_{\min}^*,y_{\max}^*]$ is marked by red dashed lines. Lower plot: the recommended concentration corridor $[x_{\min}^*,x_{\max}^*]$ is marked by red dashed lines.
}\label{fig:closed_loop_case_2}
\end{figure}
\begin{figure}[t]
\centering 
\includegraphics[width=0.9\linewidth]{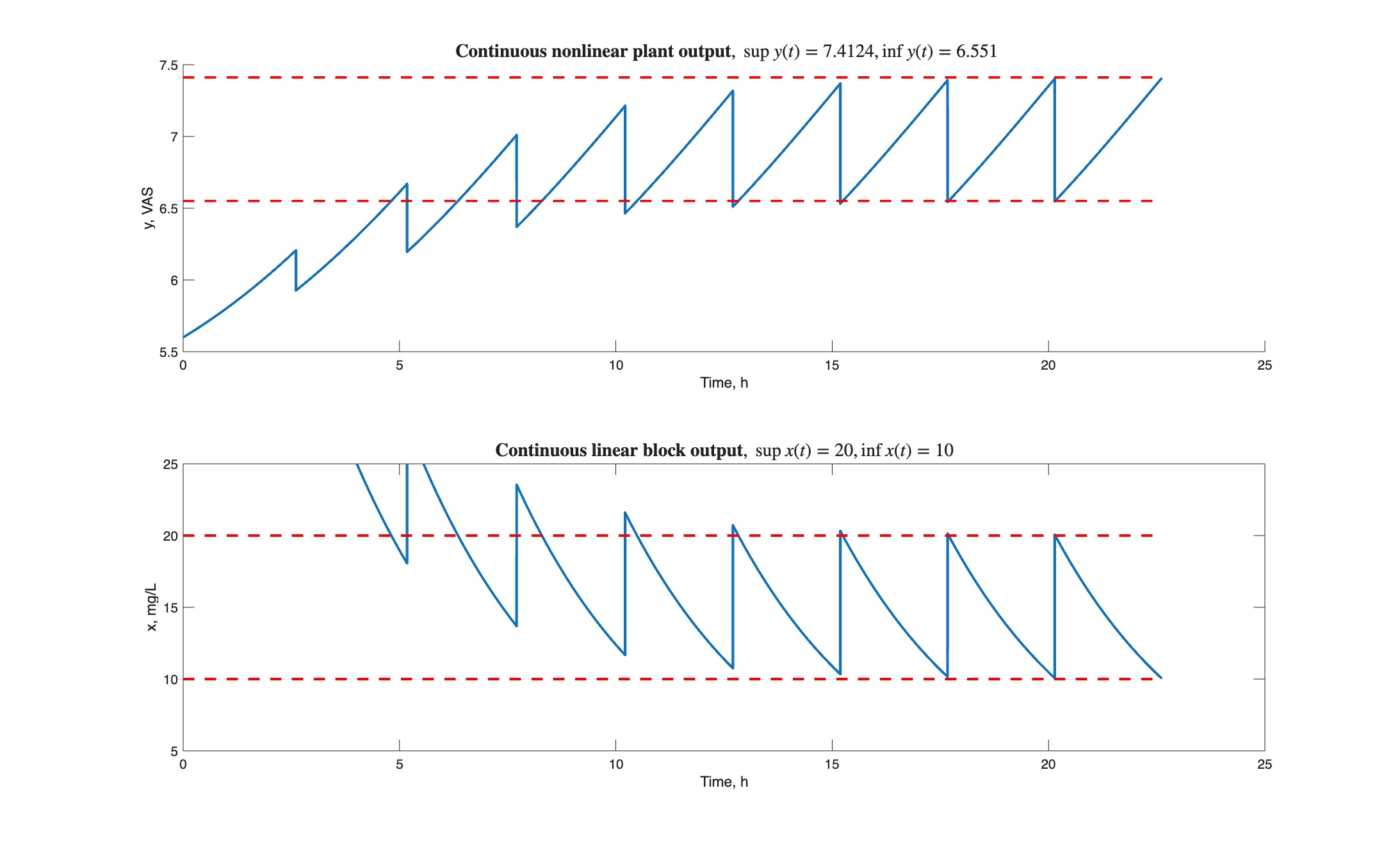}
\caption{ Closed-loop dosing of paracetamol, Case~3. After initial  IV bolus dose ($\lambda_0=\SI{2000}{mg}$), the pulse-modulated feedback drives the drug concentration and medication effect  (in blue) into desired corridor.  Upper plot: the recommended output value corridor $[y_{\min}^*,y_{\max}^*]$ is marked by red dashed lines. Lower plot: the recommended concentration corridor $[x_{\min}^*,x_{\max}^*]$ is marked by red dashed lines.
}\label{fig:closed_loop_case_3}
\end{figure}

\bibliographystyle{IEEEtran}
\bibliography{observer,refs}

\end{document}